\documentclass[12pt]{amsart} %
\usepackage[colorlinks,bookmarksopen,bookmarksnumbered,citecolor=red,urlcolor=red]{hyperref}
\usepackage{amsmath, amssymb, amsfonts, amsbsy, latexsym}
\usepackage{multirow}
\usepackage{epsfig}

\usepackage{mathptmx}       
\usepackage{helvet}         
\usepackage{courier}        
\usepackage{type1cm}        
\usepackage{graphicx}        
\usepackage{multicol}        
\usepackage[bottom]{footmisc}
\newtheorem{theorem}{Theorem}[section]
\newtheorem{corollary}[theorem]{Corollary}
\newtheorem{lemma}[theorem]{Lemma}
\newtheorem{remark}[theorem]{Remark}

\title[Spectral Fractional Laplacian with Inhomogeneous Dirichlet Data]{Spectral Fractional Laplacian with Inhomogeneous Dirichlet Data: Questions, Problems, Solutions}
 \author
 {Stanislav Harizanov}
\address{Institute of Information and Communication Technologies\\
Bulgarian Academy of Sciences\\
Acad. G. Bontchev Str., Block 25A\\
1113 Sofia, BULGARIA}
\email{sharizanov@parallel.bas.bg}
  
 \author{ Svetozar Margenov}
 \address{Institute of Information and Communication Technologies\\
Bulgarian Academy of Sciences\\
Acad. G. Bontchev Str., Block 25A\\
1113 Sofia, BULGARIA}
\email{margenov@parallel.bas.bg}

 \author{Nedyu Popivanov}
 \address{Institute of Information and Communication Technologies\\
Bulgarian Academy of Sciences\\
Acad. G. Bontchev Str., Block 25A\\
1113 Sofia, BULGARIA}
\email{nedyu@parallel.bas.bg}

\begin{document}



 \bigskip \medskip

 \begin{abstract}In this paper we discuss the topic of correct setting for
the equation $(-\Delta )^s u=f$, with $0<s <1$. The definition of
the fractional Laplacian on the whole space $\mathbb R^n$, $n=1,2,3$ is
understood through the Fourier transform, see, e.g., Karniadakis et.al. (arXiv, 2018). The
real challenge however represents the case when this equation is
posed in a bounded domain $\Omega$ and proper boundary conditions
are needed for the correctness of the corresponding problem. Let us
mention here that the case of inhomogeneous boundary data has been
neglected up to the last years. The reason is that imposing nonzero
boundary conditions in the nonlocal setting is highly nontrivial.
There exist at least two different definitions of fractional Laplacian, and there is still
ongoing research about the relations of them. They are not equivalent.
The focus of our study is a  new characterization of the spectral fractional Laplacian.
One of the major contributions concerns the case when the right hand side $f$ is
a Dirac $\delta$ function.
For comparing the differences between the solutions in the spectral and Riesz formulations, we consider an inhomogeneous fractional Dirichlet problem.
The provided theoretical  analysis is supported by model numerical
tests. 
\end{abstract}

\maketitle



\section{Introduction}\label{sec:intro}

\subsection{Riesz formulation}\label{subsec:11}
In the case of definition based on the Riesz potential (”Riesz formulation”) the
fractional Laplacian operator is introduced as below. For
$s \in (0,1)$, it is defined as

\begin{equation}\label{equation-FL}
 \left( -\Delta\right)^s u\left( x \right) =
 C\left( n, s \right) P.V. \int_{\mathbb{R}^n} \frac{u \left( x\right) - u \left( y \right)}{|x-y|^{n+2s}} dy,
\end{equation}
where $C\left( n, s\right)$ is a normalized constant.

In this setting, the corresponding homogeneous boundary value problem is:
\begin{equation} \label{eq:Riesz Laplace}
(-\Delta)^s u = f \quad in \quad \Omega\subset\mathbb R^n,\qquad u = 0 \quad in \quad \mathbb{R}^n \backslash \Omega. 
\end{equation}
For more information see \cite{claudia,bucur,oton,acosta,Gunzburger,Tersian}.

\subsection{Spectral formulation}

Unlike the Riesz definition \eqref{equation-FL} of the operator $(-\Delta)^s$, in its spectral formulation for zero Dirichlet boundary conditions we have the following (see \cite{Caff2,Caff}):
\begin{equation}    \label{eq6}
(-\Delta_{\Omega,0})^s u(x) := \sum_{k=1}^{\infty} (\lambda_k)^s (u, e_k)_{L_2
(\Omega)} e_k(x),
\end{equation}
where $\lambda_k$ and $e_k(x)$ are the corresponding eigenvalues and
eigenfunctions of the classical Dirichlet problem for the Laplacian:
\begin{equation}    \label{eq7}
-\Delta e_k = \lambda_k e_k \quad in \quad \Omega ,\qquad
e_k = 0 \quad on \quad \partial \Omega .
\end{equation}

In this setting, the related non-local elliptic problem is:
\begin{equation}    \label{eq:Spectral Laplace}
(- \Delta_{\Omega,0})^s u = f \quad in \quad \Omega, \qquad u|_{\partial \Omega} = 0.
\end{equation}

Note that, both problem formulations \eqref{eq:Riesz Laplace} and \eqref{eq:Spectral Laplace} are well-posed, regardless the difference that, in the first case the data are given in the whole domain $\mathbb{R}^n\setminus\Omega$, while in the second case - only on the boundary $\partial\Omega$. The main goal of the paper is to analyze the behavior of the exact solutions (which are in general different for the two approaches) for various fractional powers $s\in(0,1)$. Two examples are considered. The first one deals with constant right-hand-side $f$ and the solutions exhibit interface layers, due to the homogeneous Dirichlet boundary conditions. The steepness of the layers strongly depends on the problem formulation (\eqref{eq:Riesz Laplace} or \eqref{eq:Spectral Laplace}) and on the value of $s$. The second one deals with Dirac delta right-hand-side and the analysis here is based on inhomogeneous Dirichlet boundary conditions, since we force the spectral solution to coincide with the Riesz one on the boundary. Because of the singularity in the right-hand-side, the solutions also exhibit singularities. Such lack of regularity disable the usage of classical analysis in this setup.

The paper is organized as follows.

In Section~\ref{sec:2} we discuss the difference between solutions in the simple (but important) 1D case for both formulations, when the right-hand-side is $f \equiv 1$ on $(-1,1)$. One of the main
differences is in the behavior of both solutions around the boundary
of the domain. More precisely, the solution of the ''Riesz
formulation'' behaves like $[ dist (x, \partial \Omega)]^s$ around
$\partial \Omega$ (see \cite{oton}), but in the ''spectral formulation''  the behavior of
solution is quite different (see \cite{Caff}). The considered example highlights some general results of Caffarelli and Stinga. Furthermore, an open problem regarding the boundary layer asymptotic in the spectral case for $s=1/2$ is formulated.

In Section~\ref{sec3} we compare again the solutions
in both cases (Riesz and Spectral) but for the right-hand side the
Dirac $\delta_0$ function, concentrated at the origin $(0, ... , 0)$. Note that, this is a quite delicate setup, since the distribution $\delta_0$ does not belong to the classical functional spaces. We use here the definition of the inhomogeneous Dirichlet  spectral
problem (see \cite{antil, anna}). All necessary definitions and some short explanations are provided. Also some comparison between the two solutions is given.

In Section~\ref{sec:numtests} the derived theoretical results are numerically studied. 

\section{Homogeneous Dirichlet conditions}\label{sec:2} 
In this section, we consider the case of right-hand-side $f
\equiv 1 $. Let $\Omega $ be the unit ball $B_{1} = \{ \|x\| < 1 \} $, where we consider the Euclidean distance in $\mathbb{R}^n$. The corresponding solution of  the "Riesz formulation" \eqref{eq:Riesz Laplace} (see
\cite{oton}) is :
\begin{equation}    \label{eq3}
u_s^R (x) = c(n,s) (1 - \|x\|^2)^s, \qquad ||x|| < 1,
\end{equation}
where
$$c(n,s)=\frac{2^{-2s}\Gamma(n/2)}{\Gamma((n+2s)/2)\Gamma(1+s)}.$$
 Obviously, the behavior of the solution
around the boundary is
\begin{equation}\label{eq:Boundary Riesz}
u_s^R(x) \sim [dist(x, \partial B_1)]^s.
\end{equation}
It is clear that $u_s^R\in C^s$ near the boundary but it is not in $C^\alpha$ for any $\alpha>s$.

Quite different is the situation for the "spectral formulation" \eqref{eq:Spectral Laplace}. Here, we prefer for simplicity to focus only on the 1D case, which we investigate in detail. 

The corresponding  eigenvalues and
orthonormalized eigenfunctions of the Laplace operator $-\Delta$ (see \eqref{eq7}) are:
\begin{equation}    \label{eq9}
\lambda_k = \left( \frac{k \pi}{2} \right)^2, \quad e_k(x) = \sin\left[
\frac{k \pi}{2} (x+1)\right], \quad k = 1,2...
\end{equation}
Therefore, equation \eqref{eq6} reads as:
\begin{equation}    \label{eq10}
(-\Delta_{\Omega,0})^s u := \sum_{k=1}^{\infty} \left(\frac{k \pi}{2}
\right)^{2s} (u, e_k)_{L_2 (B_1)} e_k(x).
\end{equation}
If
\begin{equation}    \label{eq11}
f \equiv 1 \equiv \sum_{k=1}^{\infty} (1, e_k)_{L_2 (B_1)} e_k(x)
\equiv \sum_{m=0}^{\infty} \frac{4}{(2m + 1) \pi} e_{2m + 1}(x),
\end{equation}
then, comparing \eqref{eq10} with \eqref{eq11}, it follows that $ (u, e_k) = 0$ for $k=2m$ and
\begin{equation*}
(u, e_{2m + 1}) =
2\left(\frac{2}{\pi}\right)^{2s+1}\frac{1}{(2m+1)^{2s+1}}, \;\; m =
0,1,2...
\end{equation*}
Thus,
\begin{equation}    \label{eq12}
u_s(x)= 2\left(\frac{{2}}{\pi}\right)^{2s+1} \sum_{m=0}^{\infty}
\frac{1}{(2m+1)^{2s+1}} \sin \left[ \frac{(2m+1) \pi}{2}  (x+1)
\right].
\end{equation}
Let us compare the "solutions" in the two cases. For the "spectral formulation", the behavior of the solution \eqref{eq12}
around the boundary $|x| = 1$ is quite different than \eqref{eq3} for the "Riesz
formulation" (see Section~\ref{subsec:11}). Also, we give this simple example to illustrate the following
Caffarelli - Stinga result (see \cite{Caff})
\begin{theorem}\label{thm:Caffarelli}
(Boundary regularity for $f$ in $C^{\alpha}$ - Dirichlet). Assume
that $\Omega$ is a bounded domain and that $f \in C^{0,
\alpha}(\bar{\Omega})$, for some $0 < \alpha < 1$. Let $u$ be a
solution to \eqref{eq:Spectral Laplace}.

(a) Suppose that $0 < \alpha + 2s < 1$, $\Omega$ is a $C^1$ domain.
Then
$$u(x) \sim dist(x, \partial \Omega)^{2s}+v(x), \;\; \text{for}\;\;x\;\; \text{close to} \;\; \partial \Omega,$$
where $v \in C^{0,\alpha + 2s }(\overline{\Omega})$

(b) Suppose that $s > 1/2,\, 1< \alpha +2s <2,$ $\Omega$ is a
$C^{1,\alpha + 2s -1}$ domain. Then
$$u(x) \sim dist(x, \partial \Omega)+v(x), \;\; \text{for}\;\; x \;\; \text{close to} \;\; \partial \Omega,$$
where $v \in C^{1,\alpha + 2s-1 }(\overline{\Omega}).$

(c) If $s = \frac{1}{2}$ then
$$u(x) \sim dist (x, \partial \Omega)| \ln \ dist(x, \partial \Omega)| + w(x), \quad for \ x \ close \ to \ \partial \Omega,$$
where $w \in C^{1, \alpha} (\overline{\Omega})$.

In both cases, if $f(x_0) = 0$ for some $x_0\in\partial\Omega$, then $u(x_0) = v(x_0)$ (resp. $u(x_0) = w(x_0)$) and $u$ has the same regularity as $v$ (resp. $w$) at $x_0\in\partial\Omega$.
\end{theorem}

\subsection{Analysis of solution \eqref{eq12}}\label{subsec:21}

Because of the
symmetry, it is enough to study the behavior of the solution
\eqref{eq12} only around $x=-1$. Actually, the case $\mathbf{a) \ s>\frac{1}{2}} $ is obvious,
because $u_s \in C^1(\overline{\Omega})$. Indeed, 
\begin{equation*}
\left| \sin \left[\frac{(2m+1) \pi}{2} (x+1)\right] \right| \leq \frac{(2m+1)
\pi}{2} (x+1)
\end{equation*}
and, thus
\begin{equation*}
(1+x)^{-1} |u_s (x)| \leq 2\left(\frac{{2}}{\pi}\right)^{2s+1}  \sum_{m=1}^\infty
\frac{1}{(2m+1)^{2s}} < +\infty.
\end{equation*}
In the case $\mathbf{b) \ 0< s < \frac{1}{2}}$ we have: for any
$\epsilon, \ 0< \epsilon < 2s $
\begin{align*}
&(1+x)^{-2s+ \epsilon} \left| \sin \left[ (2m+1) \frac{\pi}{2} (x+1) \right] \right| \leq \\
&\leq (1 +x)^{-2s + \epsilon} \left[ (1+x)(2m+1) \frac{\pi}{2} \right]^{2s- \epsilon} \left| \sin \left[ (2m+1) \frac{\pi}{2} (x+1) \right] \right|^{1-2s + \epsilon} \leq \\
&\leq \left[ (2m+1) \frac{\pi}{2} \right]^{2s - \epsilon}.
\end{align*}
Then
\begin{equation*}
(1+x)^{-2s+ \epsilon} | u_s(x)| \leq \frac{4}{\pi}
\sum_{m=0}^{\infty} \frac{1}{(2m+1)^{1+ \epsilon}} < +\infty
\end{equation*}
The most interesting case is $\mathbf{c) \ s = \frac{1}{2}}$. Now,
the solution of problem \eqref{eq:Spectral Laplace} with $f \equiv 1$
in $\Omega = (-1, 1)$ is given by \eqref{eq12}, i.e.
\begin{equation} \label{eq13}
u_{1/2} (x) = \frac{8}{\pi^2} \sum_{m=0}^{\infty}
\frac{1}{(2m+1)^2} \sin \left[(2m+1) \frac{\pi}{2} (x+1) \right].
\end{equation}
We are interested in the behavior of $u_{1/2} (x)$ around
the boundary $x= \pm 1 $. Because of the symmetry in \eqref{eq13} it
is enough to study the behavior of the function
\begin{equation} \label{eq14}
v(y) := \frac{8}{\pi^2} \sum_{m=0}^{\infty} \frac{1}{(2m+1)^2} \sin
\left[ (2m+1) \frac{\pi}{2} y \right]
\end{equation}
around $y=0$. Denoting
\begin{equation} \label{eq15}
v_{N} (y) := 2 \sum_{m=0}^{N-1} \frac{4}{\pi^2(2m+1)^2 } \sin \left[
(2m+1) \frac{\pi}{2} y \right]
\end{equation}
we find
\begin{equation}\label{eq15.2}
v_{N} '' (y) = -2 Im  \left\{  e^{i \frac{\pi}{2} y}
\sum_{m=0}^{N-1} e^{i m \pi y} \right\} = - \frac{1 - cos (N \pi
y)}{\sin(\frac{\pi}{2} y)}.
\end{equation}
Using from \eqref{eq15} that $v_{N} (0) = v_{N} ' (1) = 0 $, finally
we find
\begin{align*}
\begin{split}
&v_{N} (y) = \int_0^y \left[ \int_{\lambda}^1 \frac{1-cos(N \pi t)}{\sin(\frac{\pi}{2} t)} dt \right] d \lambda = \\
&= \int_0^y \frac{t [1 - cos(N \pi t )]}{\sin(\frac{\pi}{2}t)} dt + y
\int_y^1 \frac{1 - cos(N \pi t)}{\sin (\frac{\pi}{2} t)} dt =: v_N^1
(y) + v_N^2 (y).
\end{split}
\end{align*}
Obviously $\sin(\frac{\pi}{2} t) \geq \frac{1}{2}t, \ t \in (0,1) $
and thus $ y^{-1} v_N^1 (y) \leq 4$. From another side
\begin{equation*}
y^{-1} v_N^2 (y) \;_{\overrightarrow{y \to +0}}\; \int_0^1 \frac{1-
cos(N \pi t)}{\sin(\frac{\pi}{2} t)} dt = v_N'(0) = \sum_{m=0}^{N-1}
\frac{1}{(2m+1)} \; {}_{\overrightarrow{N \to \infty}}\; \infty.
\end{equation*}
This means that the behavior of $v(y) \equiv u_{1/2} (y-1)$
is not like $ dist(x, \partial \Omega)$, as for $s > \frac{1}{2}$.
We could prove very easy the common result of \cite{Caff} in this
case:
\begin{equation*}
\frac{v_N^2(y)}{y |\ln y |} = \frac{1}{|\ln y|} \int_y^1 \frac{1 -
cos(N \pi t)}{\sin(\frac{\pi}{2}t)} dt \leq 4.
\end{equation*}
\textbf{Open problem.} Is it possible, in the spirit of the last remark of Theorem~\ref{thm:Caffarelli}, to prove for $f \equiv 1$ a sharper estimate of the asymptotic behavior of the solution
$u_{1/2}(x)$, even though $f$ does not vanish at any boundary point? For example $u_{1/2}(x) \sim dist (x, \partial \Omega)| \ln \ dist(x, \partial \Omega)|^k$ for some real $k <1$. According to the conducted numerical experiments in Section~\ref{sec:numtests}, it seems $k=0.86$ to be enough (see Fig.~\ref{fig:Boundary Layers}), which gives rise to a slight improvement of the above general theoretical result, documented in Theorem~\ref{thm:Caffarelli}, case (c).

\section{Inhomogeneous Spectral Fractional Laplacian}   \label{sec3}
We consider both formulations, leading to the following non-local problems:
\\ \textbf{A)} The "Riesz formulation":
\begin{equation}    \label{eq:Riesz Nonhomogeneous}
(- \Delta)^s u = f \quad in \quad \Omega \subset \mathbb{R}^n,
\qquad u(x) = g(x), \quad x \in \mathbb{R}^n \backslash \Omega.
\end{equation}
\textbf{B)} The "spectral formulation":
\begin{equation}    \label{eq:Spectral Nonhomogeneous}
(- \Delta)^s u = f \quad in \quad \Omega \subset \mathbb{R}^n,
\qquad u|_{\partial \Omega} = g.
\end{equation}
Note that, unlike the homogeneous case, which has been well studied, the case $g\not\equiv0$ is less clear (see \cite{antil, anna, nicole}). Furthermore, different statements for case B), including possible singularities, are also available in the literarute (e.g., see \cite{abatangelo}).

Now, we solve both cases for $f \equiv \delta_0$, where
$\delta_0$ is the Dirac function, concentrated at the origin
$O(0,...,0)$, i.e., $<\delta_0, \varphi> = \varphi(O)$.
The utilized approach gives rise to explicit formulation of the corresponding solutions in terms of infinite power series. Therefore, in this section we will not fix the functional spaces we work at, but the interested reader can derive them from the series asymptotic. 

A fundamental solution of the equation
\begin{equation}    \label{eq3.6}
(- \Delta)^s u = \delta_0 \quad in \quad \mathbb{R}^n
\end{equation}
is (see \cite[Theorem 2.3]{bucur}):
\begin{equation}    \label{eq3.7}
u_0 (x) = a (n,s) ||x||^{-n + 2s}, \quad 2s \neq n
\end{equation}
\begin{equation*}    \label{eq3.8}
u_0 (x) = a (n,s) ln ||x||, \quad 2s=n,
\end{equation*}
where the constant $a(n, s)$ is given by
\begin{equation*}
a(n,s) = \frac{\Gamma \left(\frac{n}{2} - s \right)}{2^{2s}
\pi^{\frac{n}{2}} \Gamma(s)} , \quad 2s \neq n
\end{equation*}
(see \cite[(1.13) and (1.20)]{bucur}).

In the "Riesz case" for $2s \neq n$ the function \eqref{eq3.7} is a
solution in any bounded domain $\Omega_n \subset \mathbb{R}^n$ of the
problem:
\begin{align}   \label{eq3.9}
\begin{split}
&(- \Delta)^s u = \delta_0 \quad in \quad \Omega_n, \\
&u = a(n,s) ||x||^{-n+2s}, \quad x \in \mathbb{R}^n \backslash
\Omega_n.
\end{split}
\end{align}
\begin{remark}\label{remark:L2 loc} 
It is easy to see that
\begin{equation*}
u_0 \in L^2_{loc} (\mathbb{R}^n) \Leftrightarrow
\end{equation*}
\begin{equation*}
\textbf{a)} \ n=1, \ s > \frac{1}{4}; \quad \textbf{b)} \ n=2, \ s >
\frac{1}{2}; \quad \textbf{c)} \ n=3, \ s > \frac{3}{4}.
\end{equation*}
Indeed, $\delta_0 \in H^t (\mathbb{R}^n)$ for each $t<-n/2$, and thus for the solution of the equation  \eqref{eq3.9}  it follows: $u_0 \in L^2_{loc} (\mathbb{R}^n)$, if $t+2s\ge0  \Leftrightarrow -n/2+2s>0  \Leftrightarrow s>n/4$. In this case we can not use the usual duality between $H^s$ and $H^{-s}$ (see for example \cite{Caff}).
\end{remark}

We compare the Riesz solution \eqref{eq3.7} with the solution of the "spectral fractional" problem:
\begin{equation}    \label{3.10}
(- \Delta)^s u = \delta_0 \quad \text{in} \quad \Omega_n,
\end{equation}
where $\Omega_n \subset \mathbb{R}^n$ is some appropriate bounded domain,
with boundary conditions
\begin{equation}    \label{3.11}
u|_{\partial \Omega_n} = u_0 (x)|_{\partial \Omega_n}.
\end{equation}
The cases $\mathbf a)$ and $\mathbf b)$ in the "spectral formulation" are investigated.

For the inhomogeneous Dirichlet problem \eqref{3.10}, \eqref{3.11}, according to \cite{antil}, \cite{anna}:
\begin{equation}\label{eq:antil}
\left(-\Delta_{\Omega_n}\right)^s u:=\sum_{k=1}^{\infty}\left(\lambda_k^s(u,e_k)_{L_2(\Omega_n)}-\lambda_k^{s-1}\left(u,
\frac{\partial e_k}{\partial n}\right)_{L_2(\partial\Omega_n)}\right)e_k. 
\end{equation}
Instead of using this formula for the inhomogeneous operator, in \cite{antil} it is suggested to apply the so called ``harmonic lifting'' approach, which means: we separate our problem \eqref{3.10}-\eqref{3.11} into two different problems -- a homogeneous fractional Dirichlet problem for the operator, given by \eqref{eq6} and an inhomogeneous fractional Dirichlet problem, for which we are looking for the solution of the standard non-fractional Laplace operator with zero right-hand-side and appropriate boundary data.
In other words, first we will use formula
\eqref{eq6} for the operator $(-\Delta)^s$ to find a solution $w_s (x)$ of equation
\eqref{3.10}. Then we solve:
\begin{equation}   \label{3.12}
- \Delta v = 0 \quad in \quad \Omega_n, \qquad
 v(x) = u_0 (x), \quad x \in \partial \Omega_n
\end{equation}
in a "very weak form" in terms of \cite{antil,anna,thomas}. 

\textbf{Case I: $n=1$, $s\neq 1/2$. } We will choose now $\Omega_1 = \{ |x| <1 \}$. For a solution of
\begin{equation}    \label{3.13}
\left( -\frac{d^2}{dx^2} \right)^s w = \delta_0 \quad in \quad
\Omega_1,\qquad
w|_{x= \pm 1} = 0,
\end{equation}
plugging formula \eqref{eq9} in definition \eqref{eq6}, we get
\begin{equation*}
\begin{aligned}
(- \Delta_{\Omega_1,0})^s w &:=\sum_{k=1}^{\infty} \lambda_k^s (w, e_k)_{L_2
(\Omega_1)} e_k (x) = \delta \equiv \sum_{k=1}^{\infty} < \delta,
e_k>
e_k (x) \\
&= \sum_{m=0}^{\infty} (-1)^m \sin \left[\frac{(2m+1) \pi}{2} (x+1)\right].
\end{aligned}
\end{equation*}
Then $(w, e_{2m}) = 0$ and
\begin{equation*}
(w, e_{2m+1})_{L_2 (\Omega_1)} = (-1)^m \left( \frac{(2m+1) \pi}{2}
\right)^{-2s} .
\end{equation*}
Thus the solution of \eqref{3.13} is:
\begin{equation}    \label{3.15}
w_{1,s} (x) = \left(\frac{2}{\pi}\right)^{2s}\sum_{m=0}^{\infty} \frac{(-1)^m}{(2m+1)^{2s}} \sin \left[\frac{(2m+1) \pi}{2} (x+1)\right].
\end{equation}
Since $(-1)^m=\sin[(2m+1) \pi/2]$, we rewrite \eqref{3.15} as
\begin{equation}\label{3.15.2}
w_{1,s} (x) = \left(\frac{2}{\pi}\right)^{2s}\sum_{m=0}^{\infty} \frac{1}{(2m+1)^{2s}} \cos \left[\frac{(2m+1) \pi x}{2}\right]. 
\end{equation}
\begin{theorem}\label{thm:1D Dirac}
The solution $w_{1,s}(x)$ from \eqref{3.15} possesses the following properties:
\begin{itemize}
\item[(a)] \hspace{2pt} $w_{1,s} \in L_2 (\Omega_1)\quad\Longleftrightarrow \quad s\in(1/4,1)$.
\item[(b)] \hspace{2pt} $\displaystyle w_{1,s}(0)=\left(\frac{2}{\pi}\right)^{2s}\sum_{m=0}^{\infty} \frac{1}{(2m+1)^{2s}}<+\infty\quad\Longleftrightarrow\quad s>\frac12$. 
\item[(c)] \hspace{2pt} $w_{1,s}\in C\left(\bar\Omega_1\setminus \{0\}\right),\;\forall s\in (0,1)\setminus \{1/2\}$; $w_{1,s}\in C\left(\bar\Omega_1\right),\;\forall s\in(1/2,1)$.
\end{itemize}
\end{theorem}
\begin{proof}
The function $w_{1,s} \in L_2 (-1, 1)$ iff the series
$\sum_{m=0}^{\infty} (2m+1)^{-4s}$ converges, which is true iff $s >
1/4$. Furthermore, $w_{1,s} \in C[-1,1]$ iff $s>1/2$. Then, we can use the Dirichlet criteria for convergence, because (as in \eqref{eq15.2} above)
\begin{equation}\label{eq:sum cosines}
\sum_{m=p}^{P-1} \cos \left[\frac{(2m+1) \pi x}{2}\right]=\frac{\sin (P\pi x)-\sin(p\pi x)}{2\sin(\pi x/2)},\quad \forall P>p\ge0.
\end{equation} 
Thus, the series \eqref{3.15.2} uniformly converges away from $x=0$. The proof is completed.
\end{proof}

The solution
of the inhomogeneous problem \eqref{3.12} is obviously $v_{1,s}=a(1,s)$,
because in \eqref{eq3.7} for both $x = \pm 1$ clearly
$\|x\| = 1$, and $u_0|_{x=\pm 1}=a (1,s)=const$.

Then the spectral fractional solution of \eqref{3.10}, \eqref{3.11} is:
\begin{equation}    \label{3.16}
u_s(x) = w_{1,s} (x) + a (1,s),
\end{equation}
and we compare both solutions from \eqref{eq3.7} and
\eqref{3.16}. Note that both solutions belong to $L_2
[-1,1]$ iff $ s> 1/4$. However, in order for $u_s\in C[-1,1]$ we need $s>1/2$.

\textbf{Case II: $n=2$, $s\in(0,1)$.} Now, we choose the domain to be the
square $\Omega_2 = \{ |x| < 1, \ |y| < 1 \}$.
The spectral inhomogeneous fractional problem \eqref{3.10}, \eqref{3.11} is:
\begin{equation}    \label{3.17}
(- \Delta_{\Omega_2})^s u \equiv \left( - \frac{\partial^2}{\partial x^2} -
\frac{\partial^2}{\partial y^2} \right)^{s}u = \delta_0 \quad in
\quad \Omega_2,
\end{equation}
\begin{align}   \label{3.18}
\begin{split}
&u|_{y= \pm 1} = a (2,s) (x^2 + 1)^{-1 + s}, \quad -1< x < 1; \\
&u|_{x= \pm 1} = a (2,s) (y^2 + 1)^{-1 + s}, \quad -1 < y < 1.
\end{split}
\end{align}
The corresponding eigenvalues and eigenfunctions for the
problem \eqref{eq7} in the rectangle $\Omega_2$ are:
\begin{equation}    \label{3.19}
\lambda_{k,m} = (k^2 + m^2) \frac{\pi^2}{4}, \quad k,m=1,2...
\end{equation}
\begin{equation}    \label{3.20}
e_{k,m} (x,y) = \sin \left[ \frac{k \pi}{2} (x+1) \right] \sin
\left[ \frac{m \pi}{2} (y+1) \right].
\end{equation}
According to \eqref{eq6}, a solution of equation \eqref{3.17}
with homogeneous Dirichlet conditions is:
\begin{equation}    \label{3.21}
w_{2,s} (x,y) = \sum_{k,m=0}^{\infty}
\frac{(-1)^{k+m}}{\lambda^s_{2k+1,2m+1}} e_{2k+1, 2m+1} (x,y)
\end{equation}
Following the 1D case, we can rewrite \eqref{3.21} as 
\begin{equation}    \label{3.21.2}
w_{2,s}(x,y)=\sum_{k,m=0}^{\infty}
\frac{(2/\pi)^{2s}}{[(2k+1)^2 + (2m+1)^2]^s} \cos \left[ \frac{(2k+1) \pi x}{2} \right] \cos
\left[ \frac{(2m+1) \pi y}{2} \right].
\end{equation}
\begin{theorem}\label{thm:2D Dirac}
The solution $w_{2,s}(x,y)$ from \eqref{3.21.2} possesses the following properties:
\begin{itemize}
\item[(a)] \hspace{2pt} $w_{2,s} \in L_2 (\Omega_2)\quad\Longleftrightarrow \quad s\in(1/2,1)$.
\item[(b)] \hspace{2pt} $\displaystyle w_{2,s}(0,0)=\left( \frac{2}{\pi} \right)^s \sum_{k,m=0}^{\infty}
\frac{1}{[(2k+1)^2 + (2m+1)^2]^s}=+\infty,\quad\forall s\in (0,1)$, \\ which means that our ``spectral solution'' with homogeneous Dirichlet data, as the fundamental solution \eqref{eq3.7} in the Riesz formulation, is always unbounded.
\item[(c)] \hspace{2pt} $w_{2,s} \in C\left(\bar\Omega_2\setminus(0,0)\right)$, for $s\in(1/2,1)$.
\end{itemize}
\end{theorem}
\begin{proof}
It is easy to see that
\begin{equation*}
w_{2,s} \in L_2 (\Omega_2)\quad
\Longleftrightarrow \quad\sum_{k,m=1}^{\infty} \frac{1}{(k^2 + m^2)^{2s}} < +\infty \quad
\Longleftrightarrow \quad s\in(1/2,1),
\end{equation*}
which concludes (a). Statement (b) is straightforward.

In order to prove (c), we begin with the following well-known result:
\begin{lemma}\label{lemma1}
Let $\{\alpha_m\}_{m=0}^\infty$ be a non-increasing, non-negative sequence and $\{\beta_m(y)\}_{m=0}^\infty$ be a sequence of continuous functions, which partial absolute sums are uniformly bounded, i.e., there exists a constant $L$, such that if $B_M(y):=\sum_{m=0}^{M}\beta_m(y)$, then $|B_M(y)|\le L$, for all $M\in\mathbb N$ and $y$. Then, the following estimate holds true:
$$
\left|\sum_{m=0}^M\alpha_m\beta_m(y)\right|\le\alpha_0 L,\qquad \forall M\in\mathbb N.
$$
\end{lemma}

Now, we apply Lemma~\ref{lemma1} to the solution \eqref{3.21.2}. Let $(x_0,y_0)\in\bar\Omega_2\setminus(0,0)$. Due to symmetry, without loss of generality let $y_0\neq0$. Take a local neighborhood $\mathcal N_{(x_0,y_0)}$, which is in $\Omega_2\setminus\{y=0\}$, respectively $\bar\Omega_2\setminus\{y=0\}$, if $(x_0,y_0)\in\Omega_2$, respectively $(x_0,y_0)\in\partial\Omega_2$. We will prove uniform convergence of the series \eqref{3.21.2} in $\mathcal N_{(x_0,y_0)}$ with respect to the following definition: for every $\varepsilon>0$, there exists an integer $K=K(\varepsilon)$ such that for every $K_1>K$ the ``partial series''
$$
\left|\sum_{2K^2\le k^2+m^2 < 2K^2_1}
\frac{1}{[(2k+1)^2 + (2m+1)^2]^s} \cos \left[ \frac{(2k+1) \pi x}{2} \right] \cos
\left[ \frac{(2m+1) \pi y}{2} \right]\right|<\varepsilon,
$$
for all $(x,y)\in\mathcal N_{(x_0,y_0)}$.  Indeed, since $y\neq0$, equation \eqref{eq:sum cosines} gives rise to
$$
\left|\sum_{m=p}^{P-1} \cos \left[\frac{(2m+1) \pi y}{2}\right]\right|\le\frac{1}{\sin(\pi y/2)}\quad \forall P>p\ge0,
$$
and we can choose a constant $L<+\infty$, such that $1/\sin(\pi y/2)\le L$, for all $(x,y)\in\mathcal N_{(x_0,y_0)}$. For a fixed $k$ we set
$$
\alpha_m:=\frac{1}{[(2k+1)^2 + (2m+1)^2]^s}, \qquad \beta_m(y):=\cos \left[\frac{(2m+1) \pi y}{2}\right].
$$
Obviously the assumptions in Lemma~\ref{lemma1} are satisfied for
our choice. We consider three separate cases. Firstly, let
$\sqrt{2}K\le k\le \sqrt{2}K_1$. Then $m\ge0$ and by
Lemma~\ref{lemma1}
\begin{equation} \label{33estimate}
\begin{split}
&\left|\sum_{0\le m <\sqrt{2K_1^2 - k^2}}\frac{1}{[(2k+1)^2 + (2m+1)^2]^s} \cos\left[\frac{(2m+1)\pi y}{2}\right]\right| \\
&=\left|\sum_{0\le m <\sqrt{2K_1^2 - k^2}}\alpha_m\beta_m(y)\right|\le\alpha_0L< \frac{L}{(2k+1)^{2s}}.
\end{split}
\end{equation}
Since $\sum_k (2k+1)^{-2s}<+\infty$ for $s>1/2$, we can choose $K$ such that
$$\sum_{k=K}^{\infty}\frac{L}{(2k+1)^{2s}}<\frac{\varepsilon}{3}.$$
Using the estimate \eqref{33estimate} it follows
$$
\sum_{\sqrt{2}K\le k\le \sqrt{2}K_1} \left| \cos \left[ \frac{(2k+1)
\pi x}{2} \right] \right| \left|\sum_{m}\alpha_m\beta_m(y)\right|
\le \sum_{k=K}^{2K_1}\frac{L}{(2k+1)^{2s}}
$$

This guarantees that the corresponding part of the ``partial series'' within the region $\sqrt{2}K\le k\le \sqrt{2}K_1$ is less than $\varepsilon/3$.

Secondly, let $K\le k\le \sqrt{2}K$. Then in this case for the fixed $k$  we have $2K^2\le k^2+m^2 < 2K^2_1$, i.e. now $m\ge m_k:=[\sqrt{2K^2-k^2}]+1$, or  $m\ge m_k:=\sqrt{2K^2-k^2},$ if the last number is an integer. In both cases
\begin{equation*}
\begin{split}
&\left|\sum_{m_k \le m <\sqrt{2K_1^2-k^2}}\frac{1}{[(2k+1)^2 + (2m+1)^2]^s} \cos\left[\frac{(2m+1)\pi y}{2}\right]\right| \\
&=\left|\sum_{m_k \le m <\sqrt{2K_1^2-k^2}}\alpha_m\beta_m(y)\right|
\le \alpha_{m_k} L< \frac{L}{(2k+1)^{2s}}.
\end{split}
\end{equation*}
As in the first case it follows $ |\sum|<\varepsilon /3$   with the
same choise of $K$. This guarantees that the corresponding part of
the ``partial series'' within the region $K\le k\le \sqrt{2}K$  is
less than $\varepsilon/3$.

Third, let $0\le k\le K$. Then, since $k^2+m^2\ge2K^2$ within the
region of interest, $m\ge K$, and analogously to the second case we
conclude
$$
\left|\sum_{\sqrt{2K^2-k^2}\le m <\sqrt{2K_1^2-k^2}}\frac{\cos[(2m+1)\pi y/2]}{[(2k+1)^2 + (2m+1)^2]^s}\right|\le\alpha_K L<\frac{L}{(2K+1)^{2s}}.
$$
Hence the ``partial series'' within the third region $0\le k\le K$ is bounded by
$$
\sum_{k=0}^K\frac{L}{(2K+1)^{2s}}=\frac{L(K+1)}{(2K+1)^{2s}}<L(2K+1)^{1-2s}<\frac{\varepsilon}{3}
$$
for large enough $K$, as  $s>1/2$. Taking the value of K bigger than the values in the first and the third case the proof is completed.
\end{proof}

To solve the corresponding inhomogeneous problem \eqref{3.17},
\eqref{3.18} using the "harmonic lifting" technique we have to solve
\begin{equation}    \label{3.22}
- \Delta v_{2,s} = 0 \quad in \quad \Omega_2,
\end{equation}
\begin{align}   \label{3.23}
\begin{split}
&v_{2,s}|_{y= \pm 1} = (x^2 + 1)^{-1 + s} \quad -1<x<1 \\
&v_{2,s}|_{x= \pm 1} = (y^2 + 1)^{-1 + s} \quad -1<y<1
\end{split}
\end{align}
It follows from the symmetry that it is enough to solve \eqref{3.22}
with boundary conditions:
\begin{align*}
\begin{split}
&\widetilde v_{2,s}|_{y = \pm 1} = (x^2 + 1)^{-1 + s} - 2^{-1+s} =: \varphi_1 (x), \\
&\widetilde v_{2,s}|_{x = \pm 1} = 0,
\end{split}
\end{align*}
which classical solution is
\begin{equation*}
\widetilde v_{2,s} (x,y) = \sum_{k=1}^{\infty} A_{k} \left(\cosh \left[ \frac{k
\pi}{2}\right]\right)^{-1} \sin \left[ \frac{k \pi}{2} (x+1) \right]
\cosh \left[ \frac{k \pi}{2}y \right],
\end{equation*}
where $A_k = \int_{-1}^{1} \varphi_1 (x) \sin \left[ \frac{k \pi}{2}
(x+1) \right] dx$. Note, that $\varphi_1$ is an even function, while $\sin \left[ \frac{k \pi}{2}
(x+1) \right]$ is even for odd $k$ and odd for even $k$. Thus, $A_{2k}=0$, while $A_{2k+1}=2\int_{0}^{1} \varphi_1 (x) \sin \left[ \frac{(2k+1) \pi}{2}(x+1) \right] dx$. Then a spectral fractional solution of
\eqref{3.17}, \eqref{3.18} is:
\begin{equation}\label{eq:Spectral2DDelta}
u_s(x,y) = w_{2,s} (x,y) + a(2,s) \left[ \widetilde v_{2,s}(x,y) + \widetilde v_{2,s}(y,x) +
2^{-1+s}\right].
\end{equation}

\section{Numerical tests}\label{sec:numtests}
In this section we numerically confirm the theoretical results from Sections~\ref{sec:2}--\ref{sec3} and address various observations on the behavior of the corresponding solutions. 

For the homogeneous Dirichlet boundary problem with right-hand-side $f\equiv 1$ on $(-1,1)$, on Fig.~\ref{fig:Constant RHS} we plot the solutions with respect to both formulations, when $s=\{0.25,0.5,0.75\}$. In \eqref{eq12} we truncate the sum at $10^4$. We observe that in all the cases, the Riesz solution point-wise exceeds the spectral one everywhere in $(-1,1)$. At the endpoints $\pm 1$, of course, the two solutions preserve the boundary conditions and are zero. Furthermore, the bottom right plot in Fig.~\ref{fig:Constant RHS} illustrates the maximum of both solutions as a function of $s\in(0,1)$. This maximum is always attended at $x=0$, and we see that for the spectral formulation the maximum linearly decays as $s$ increases, while for the Riesz formulation, this maximum is a quadratic function in $s$ for $s\in(0,1/2)$, with a peak at $x=0.25$ and only for $s\in(1/2,1)$ becomes a linear function. In conclusion, we observe that when $s<1/2$ the Riesz formulation substantially differs from the spectral formulation. We also confirm that both solutions converge to those of the classical (local) problems, when $s\to 0$ and $s\to 1$, meaning that the fractional Laplacian formulations are indeed continuous extensions of the standard non-fractional one.
\begin{figure}[t]
\centering
\begin{tabular}{cc}
$s=0.25$& $s=0.5$ \\
\includegraphics[width=0.4\textwidth]{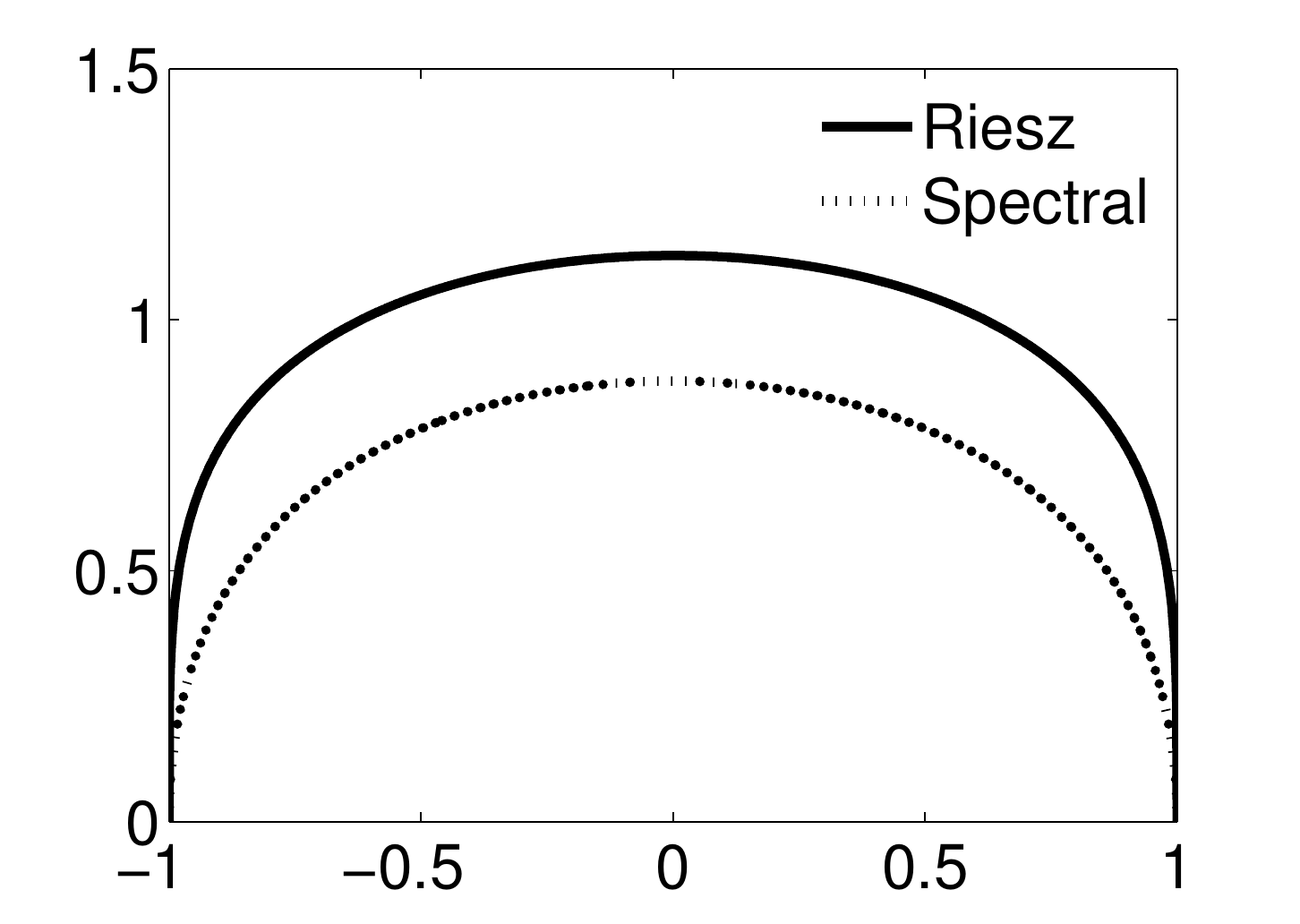} &
\includegraphics[width=0.4\textwidth]{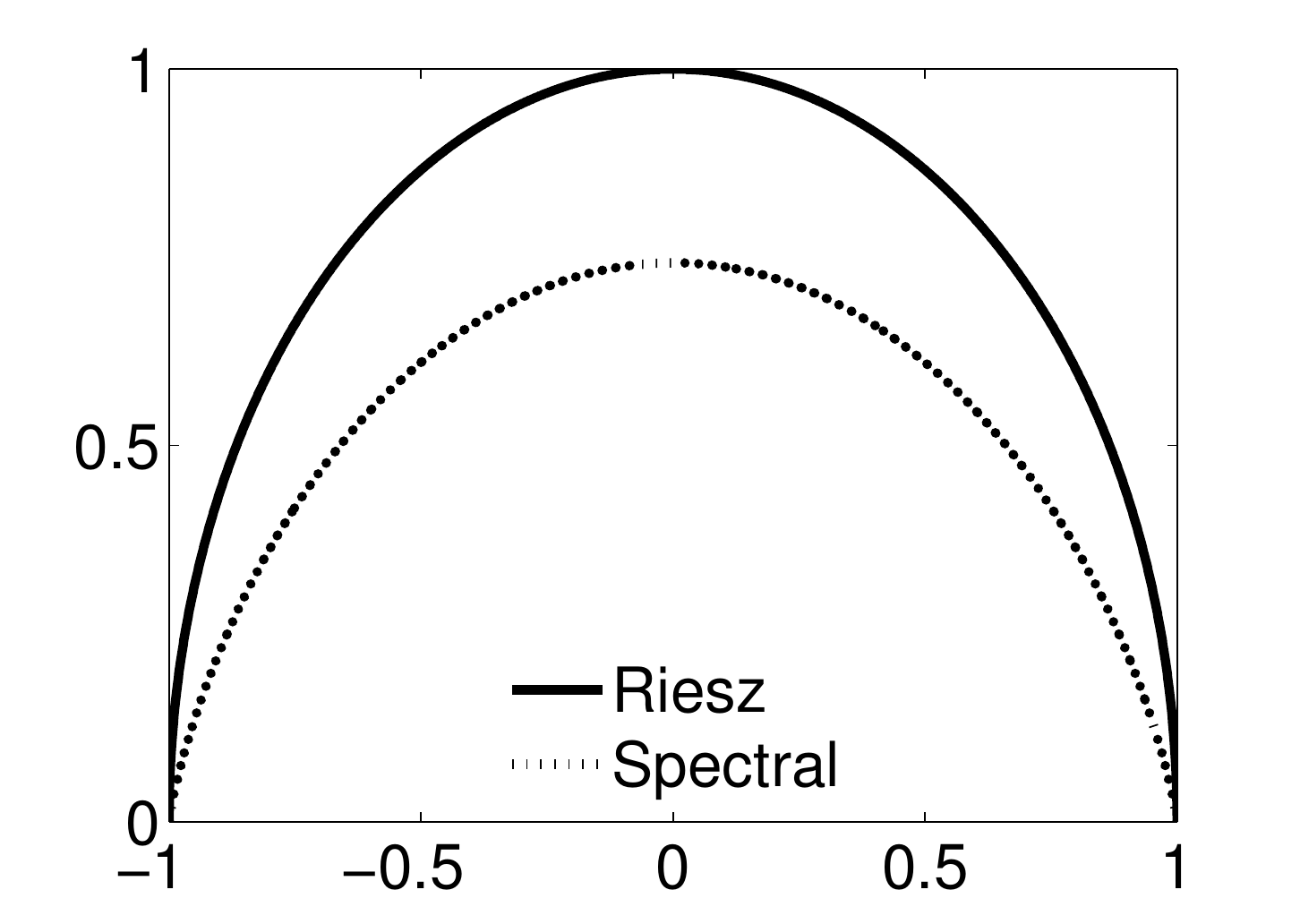} \\
$s=0.75$ & $u_s(0)$, $u_s^R(0)$, $s\in(0,1)$\\ 
\includegraphics[width=0.4\textwidth]{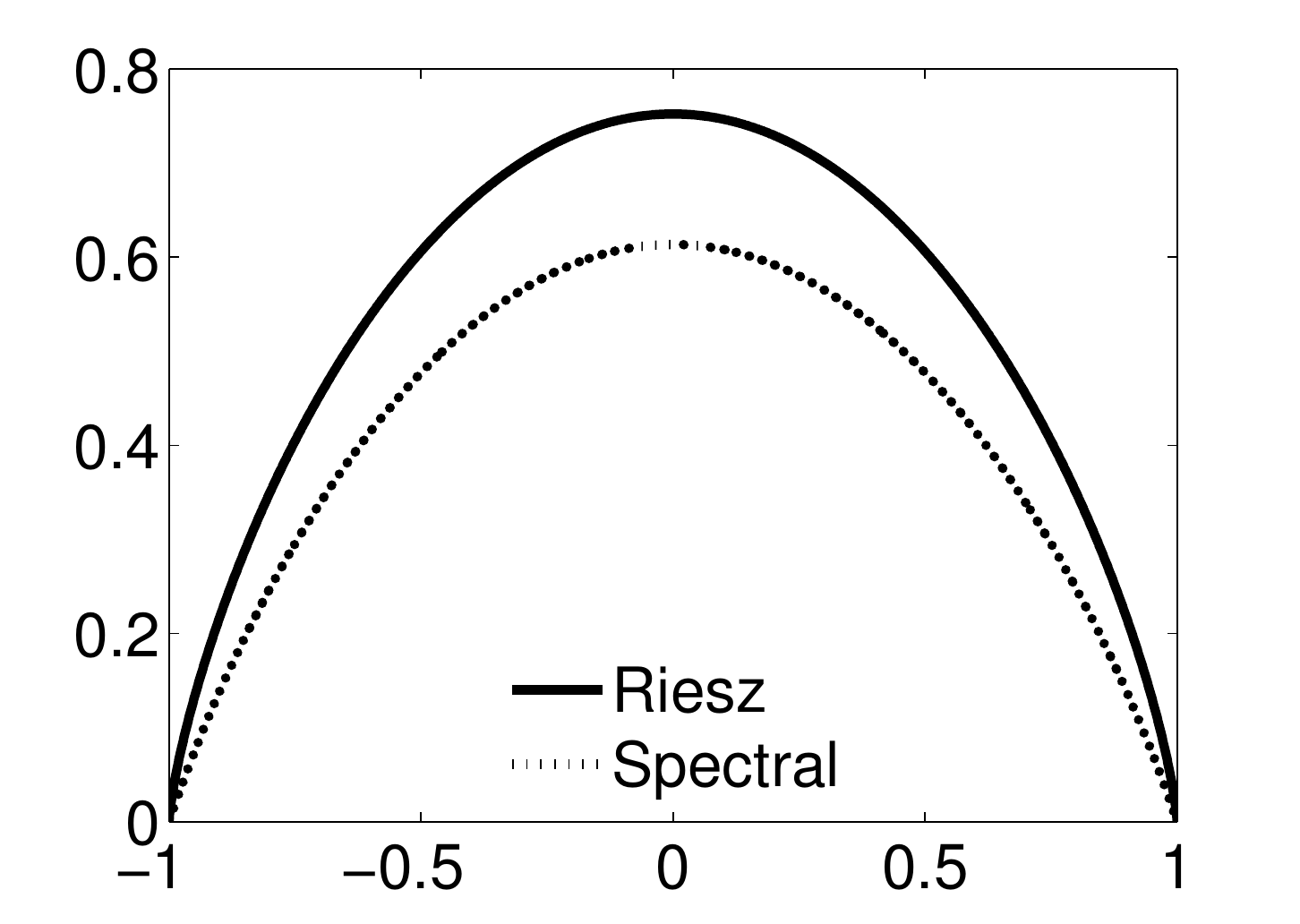} &
\includegraphics[width=0.4\textwidth]{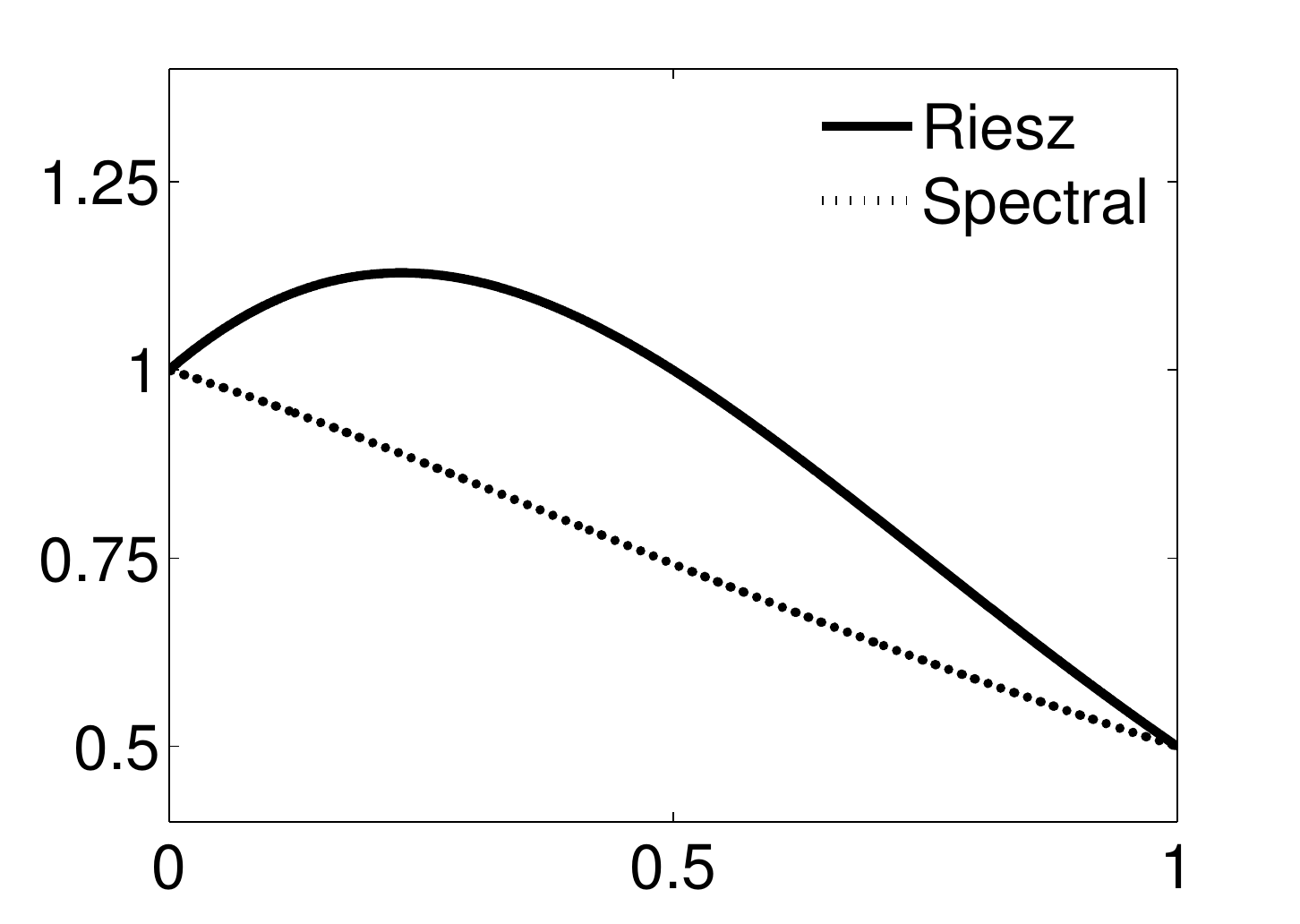} \\
\end{tabular}
\caption{Comparison of the solutions \eqref{eq3} and \eqref{eq12}. Up: The whole solutions for $s=\{0.25,0.5\}$ as a function of $x$. Down: (Left) The whole solution for $s=0.75$ as a function of $x$; (Right) The value at $x=0$ as a function of $s$.}\label{fig:Constant RHS} 
\end{figure}

\begin{figure}[t]
\centering
\begin{tabular}{cc}
Behavior at $x\to-1$& $\frac{\ln\left(|u_{1/2}(-1+j\cdot10^{-6})|/(j\cdot10^{-6})\right)}{\ln|\ln(j\cdot10^{-6})|}$\\
\includegraphics[width=0.4\textwidth]{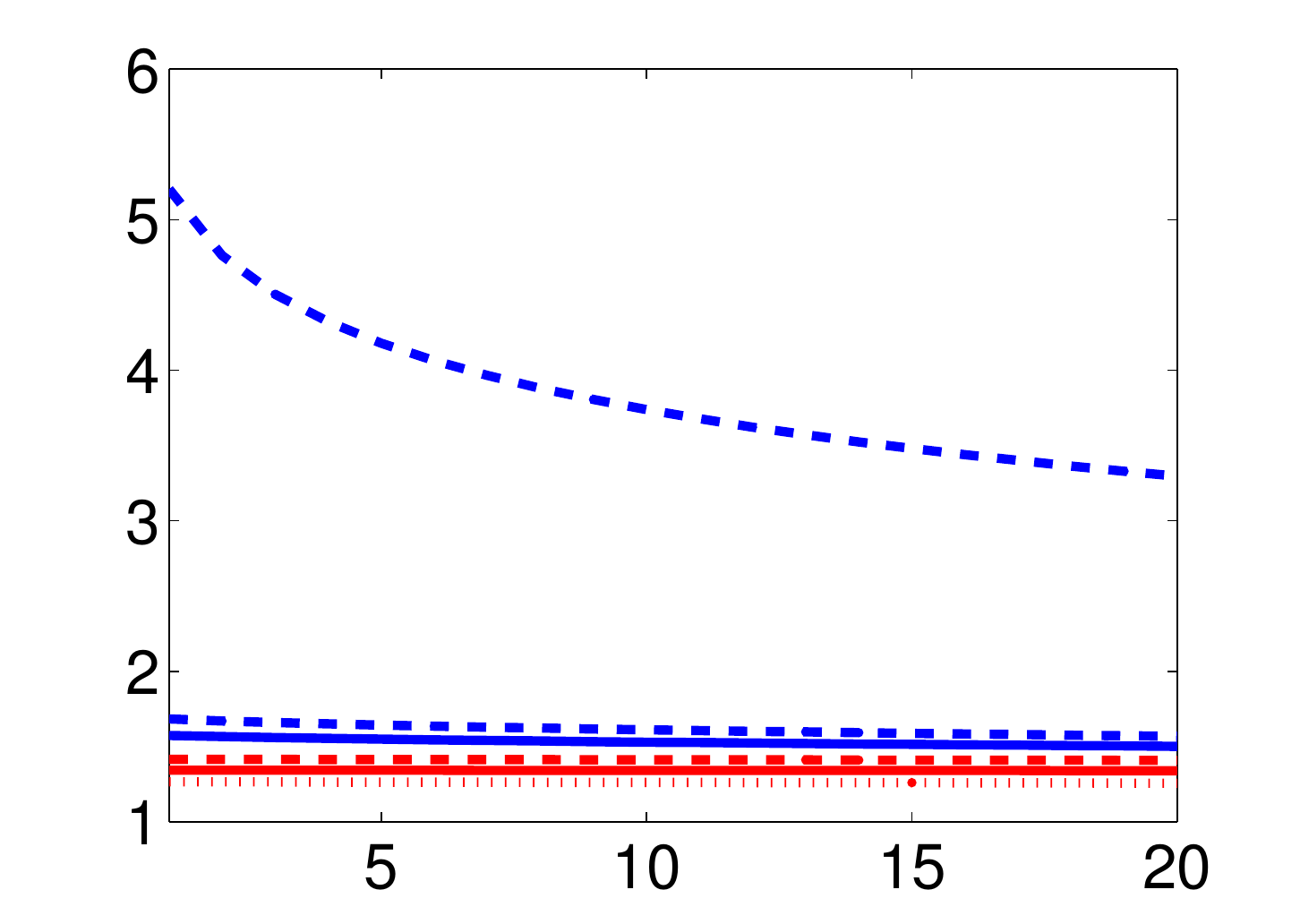} &
\includegraphics[width=0.4\textwidth]{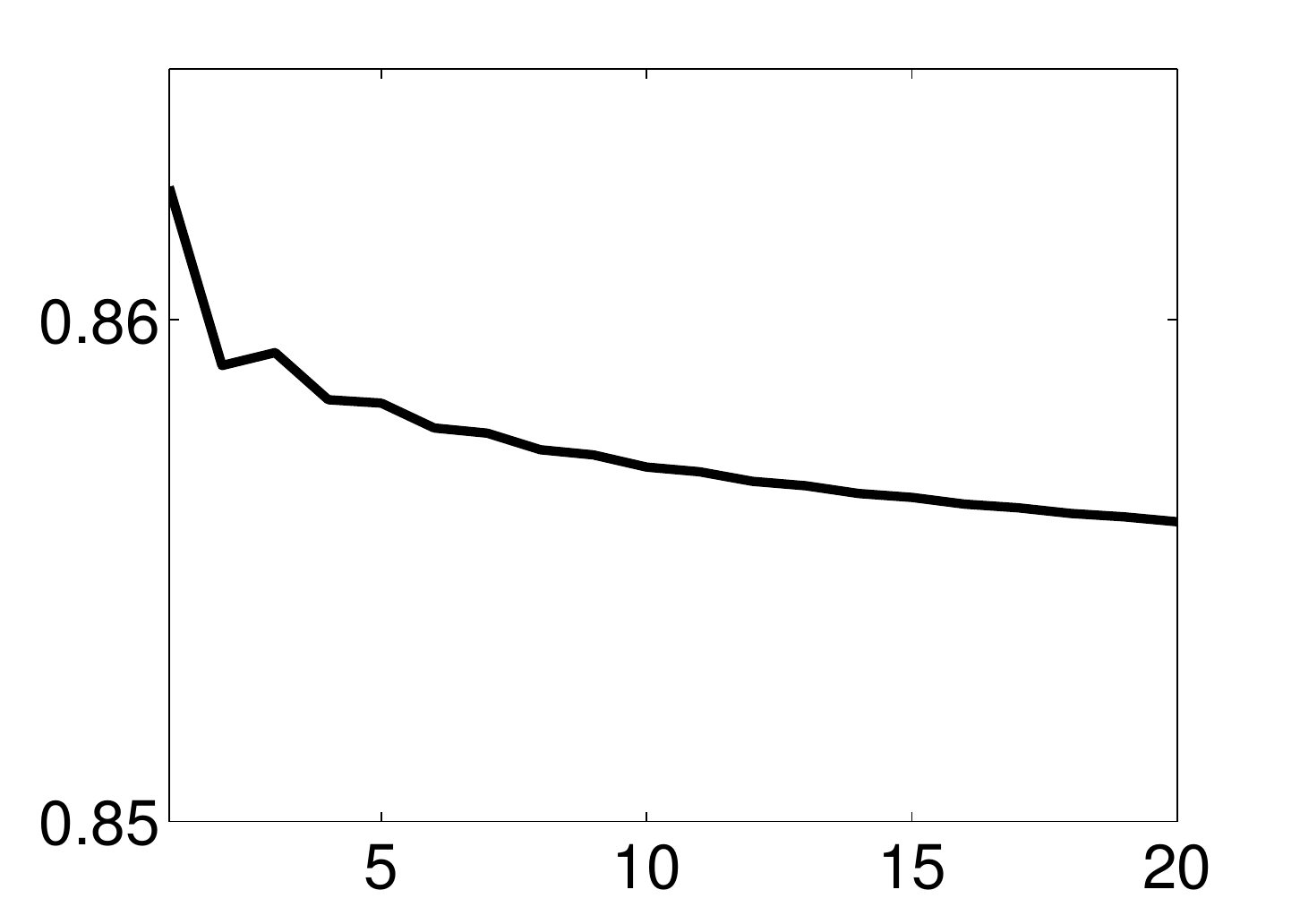}\\ 
\end{tabular}
\caption{(Left) The boundary layer behavior $u_s^R(x)/(1+x)^s$ for \eqref{eq3} vs. the boundary layer behavior $u_s(x)/(1+x)^{\min(2s,1)}$ for \eqref{eq12}, $s=0.25$ (solid lines), $s=0.5$ (dashed lines), and $s=0.75$ (dotted lines); (Right) More detailed boundary layer behavior for $u_{1/2}(x)$ for \eqref{eq12}, $x\to -1$. Both plots are with respect to the $20$ most left grid points $j=1,2,\dots,20$ on a uniform grid with $h=2^{-10}$ (left) and $h=10^{-6}$ (right). Red lines - Riesz formulation; Blue lines - spectral formulation.}\label{fig:Boundary Layers} 
\end{figure}

The left plot in Fig.~\ref{fig:Boundary Layers} deals with the steepness of the interface layers around $x=-1$ of the two solutions and aims at validating the theoretical results in \eqref{eq:Boundary Riesz} and Theorem~\ref{thm:Caffarelli}. A uniform grid on $[-1,1]$ with step size
$h=2^{-10}$ is considered and the ratios
$$
\frac{u_s^R(-1+j\cdot h)}{(j\cdot h)^s},\qquad \frac{u_s(-1+j\cdot h)}{(j\cdot h)^{\min(2s,1)}},\qquad j=\{1,\dots,20\},\;s=\{0.25,0.5,0.75\},
$$
for the Riesz and the spectral formulations, respectively, are plotted. The graphs agree with the theory. In particular, it is clearly visible that the case $s=1/2$ for the spectral formulation is the subtle one, where additional logarithmic factors are needed. The right plot is devoted to a more detailed analysis of this case, where we assume that $u_{1/2}(x)\sim (x+1)|\ln(x+1)|^k$, as $x\to -1$. Then, $k\sim \frac {\ln(u_{1/2}(x)/(x+1)}{\ln|\ln(x+1)|}$ and it can be numerically estimated. We, again, use uniform grid, but this time a much finer one as $h=10^{-6}$, and we compute the series in \eqref{eq12} with higher accuracy, considering $m\le 10^6$. The plot of the first $20$ ratios clearly indicates that $k<1$, namely $k\sim 0.85$. This is also confirmed at the original coarse grid with $h=2^{-10}$ (see Table~\ref{tab:Boundary}). Therefore, for this particular right-hand-side ($f\equiv 1$) the general result of Caffarelli - Stinga, cited in Theorem~\ref{thm:Caffarelli}, can slightly be improved.

\begin{table}[]
\centering
 \begin{tabular}{|c|cc|cc|cc|}
\hline
\multirow{3}{*}{$s$} & \multicolumn{2}{|c|}{Riesz \eqref{eq3}} & \multicolumn{2}{|c|}{Spectral \eqref{eq12}} & \multicolumn{2}{|c|}{Spectral \eqref{eq12}}\\ \cline{2-7}
& \multicolumn{2}{|c|}{$\frac{u_s^R(-1+j\cdot h)}{(j\cdot h)^s}$} & \multicolumn{2}{|c|}{$\frac{u_s(-1+j\cdot h)}{(j\cdot h)^{\min(2s,1)}}$} & \multicolumn{2}{|c|}{$\frac{u_s(-1+j\cdot h)}{j\cdot h|\ln j\cdot h|^{0.85}}$}\\ 
& min &  max & min & max & min & max \\ \hline 
$0.25$ & 1.3386 & 1.3417 & 1.5004 & 1.5718 & -- & --  \\
$0.50$ & 1.4073 & 1.4139 & 3.2960 & 5.2026 & 1.0606 & 1.0717 \\
$0.75$ & 1.2559 & 1.2647 & 1.5669 & 1.6824 & -- & --  \\\hline    
\end{tabular}
\caption{Numerical validation of the steepness of the interface layers for $j=1\dots 20$, and $h=2^{-10}$.}\label{tab:Boundary}
\end{table}

For the case of inhomogeneous fractional Laplace problem with right-hand-side $\delta_0$ we illustrate the corresponding 1D solutions with respect to both formulations (see Fig.~\ref{fig:Delta1D}) and the corresponding 2D spectral solution (see Fig.~\ref{fig:Delta2D}) for various fractional powers $s$. In 1D, we consider $s=\{0.25,0.45,0.55\}$, as only for $s>1/4$, $u_0 \in L^2_{loc} (\mathbb{R})$ due to Remark~\ref{remark:L2 loc} and $w_{1,s}\in L_2(-1,1)$ (i.e., this is the smallest meaningful value of $s$ for both formulations of the particular problem), while $s=1/2$ serves as a point of singularity for both formulations, as $u_0(0)=w_{1,s}(0)=+\infty$ when $s<1/2$ and $u_0(0)=w_{1,s}(0)=0$ when $s>1/2$. Furthermore, for the boundary case $s=0.25$ the graph of $u_0$ is monotone in $(-1,0)$, while the graph of $w_{1,1/4}$ is oscillatory. The latter oscillating behavior increases for $s<1/4$ (note that we have already proven that $w_{1,s}\in C([-1,1]\setminus \{0\})$) and disappears for $s>1/4$. Apart from that, $w_{1,s}$ and $u_0$ are quite alike. In 2D, the observations are similar. The difference is, that the oscillating behavior of $w_{2,s}$ is strongly present on the lines $x=0$ and $y=0$ and disappears only for $s>0.75$, as illustrated in Fig.~\ref{fig:Delta2D}. 

\begin{figure}
\centering
\begin{tabular}{ccc}
$s=0.25$& $s=0.45$ & $s=0.55$\\
\includegraphics[width=0.32\textwidth]{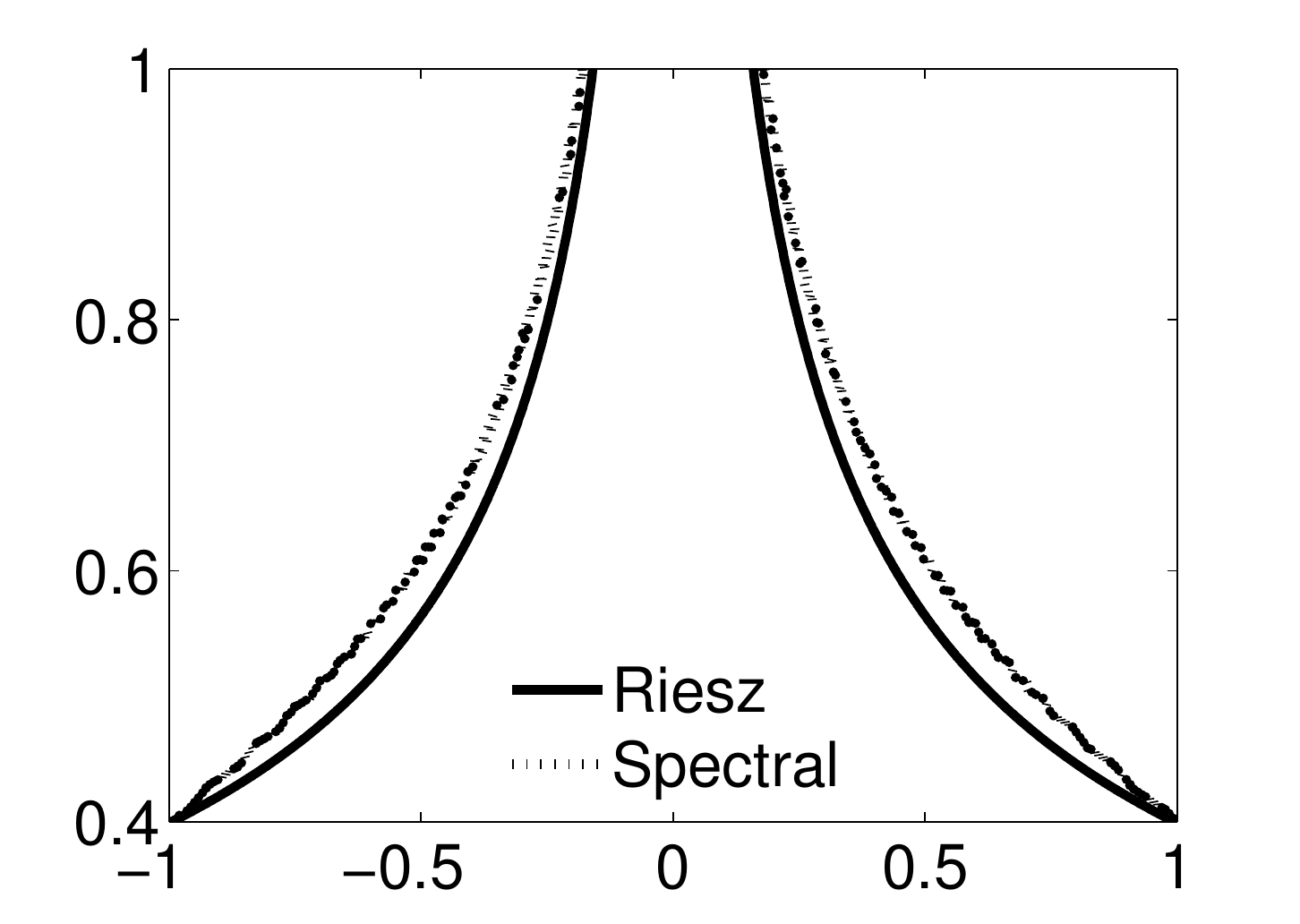} &
\includegraphics[width=0.32\textwidth]{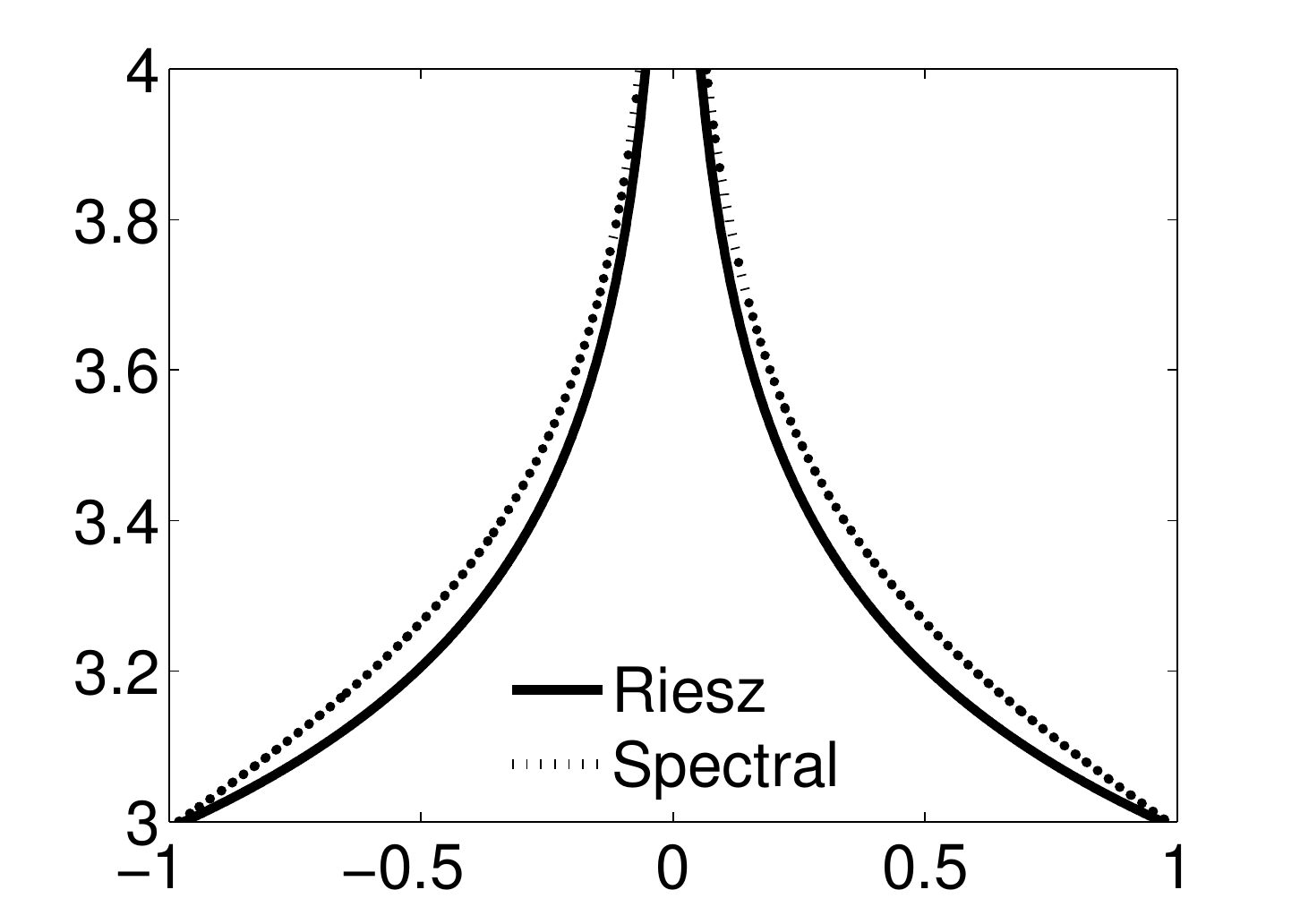} &
\includegraphics[width=0.32\textwidth]{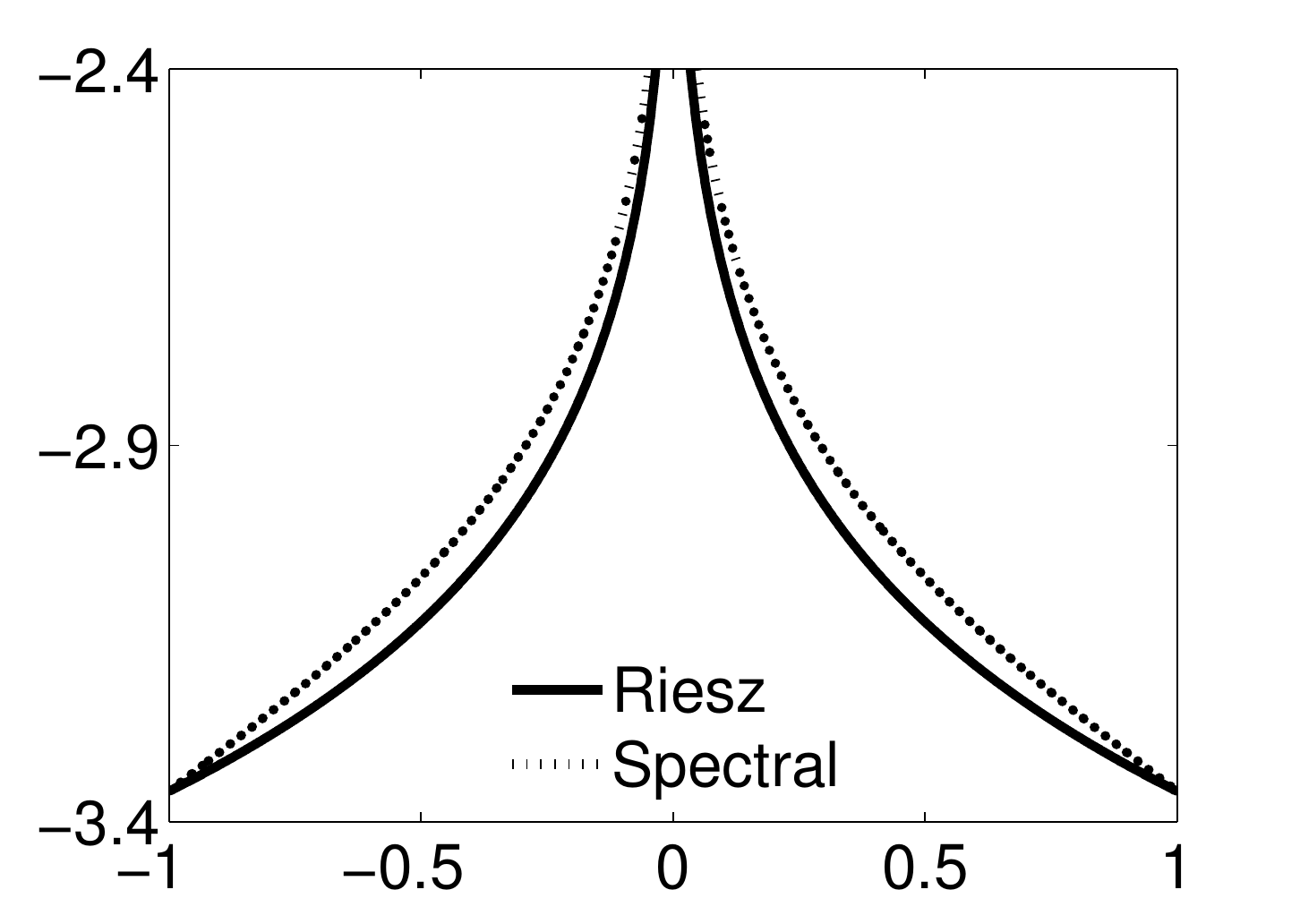}\\
\end{tabular}
\caption{Comparison of the solutions \eqref{eq3.7} and \eqref{3.16} in 1D for $s=\{0.25,0.45,0.55\}$.}\label{fig:Delta1D} 
\end{figure}

\begin{figure}
\centering
\begin{tabular}{ccc}
$w_{2,s}(x,y)$, $s=0.50$& $w_{2,s}(x,y)$, $s=0.60$ & $w_{2,s}(x,y)$, $s=0.75$\\
\includegraphics[width=0.32\textwidth]{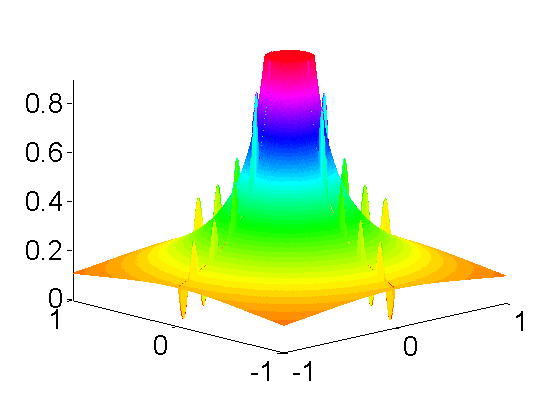} &
\includegraphics[width=0.32\textwidth]{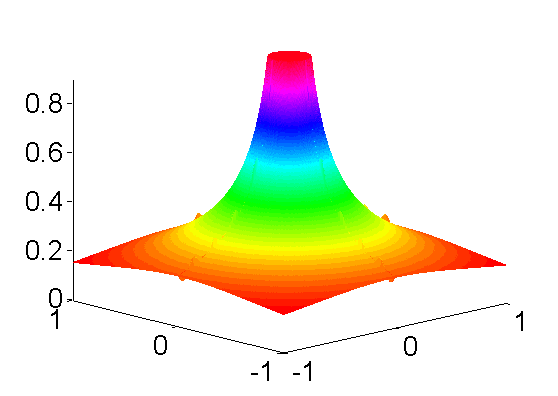} &
\includegraphics[width=0.32\textwidth]{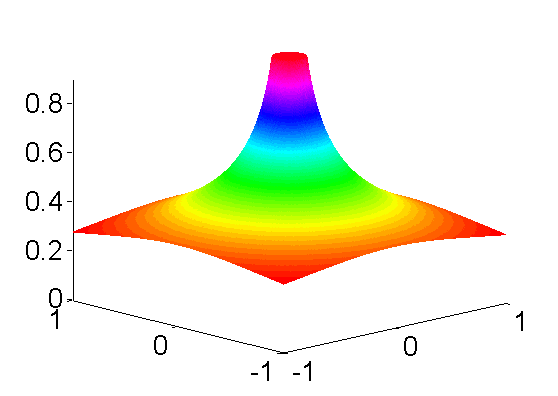}\\
$|u_0(x,y)-w_{2,s}(x,y)|$ & $|u_0(x,y)-w_{2,s}(x,y)|$ & $|u_0(x,y)-w_{2,s}(x,y)|$\\
\includegraphics[width=0.32\textwidth]{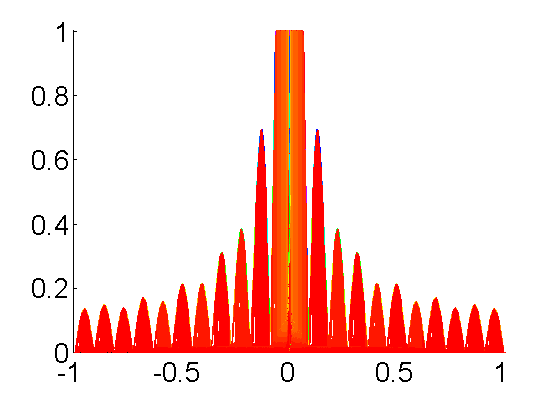} &
\includegraphics[width=0.32\textwidth]{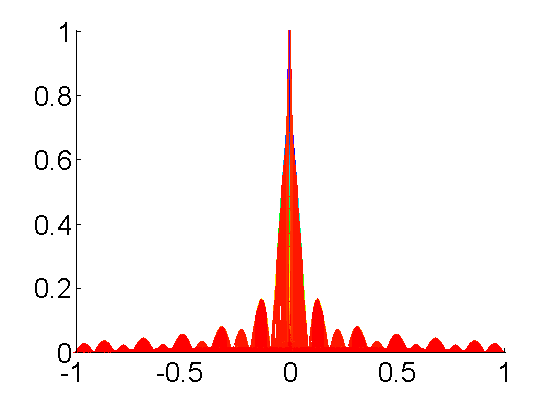} &
\includegraphics[width=0.32\textwidth]{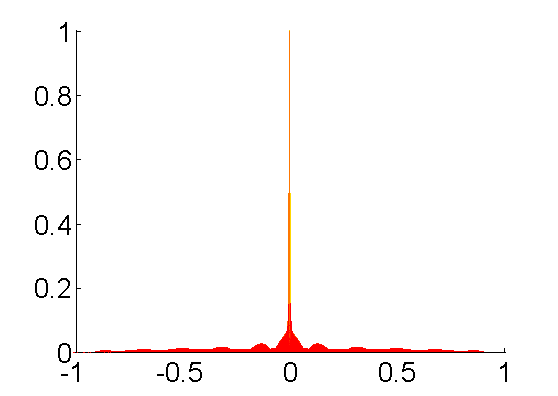}\\
\end{tabular}
\caption{Top: Visualization of the spectral solution \eqref{eq:Spectral2DDelta} for the 2D Dirac delta inhomogeneous fractional Laplace problem \eqref{3.17}, \eqref{3.18}, with $s=\{0.5,0.6,0.75\}$. Bottom: Difference image in the $(x,z)$-plane of the corresponding two solutions \eqref{eq3.7} and \eqref{eq:Spectral2DDelta}.}\label{fig:Delta2D} 
\end{figure}

\section{Concluding remarks}\label{sec:concl}
The detailed comparative analysis of the  Riesz and spectral 
formulations in the case of homogeneous boundary conditions and $f\equiv 1$ well 
demonstrates the difference between the corresponding solutions. In agreement with the 
theoretical estimates, the conducted numerical tests clearly illustrate the behavior of 
the boundary layers, additionally contributing to some better understanding of the 
Open problem formulated at the end of Section~\ref{sec:2}. The 
observation that the Riesz and spectral solutions could substantially differ 
far form the boundary layers is also an important one.

Our major theoretical contribution concerns the case of inhomogeneous boundary 
conditions when the right hand side $f$ is a Dirac $\delta$ function. Taking 
as Dirichlet data the boundary values of the fundamental Riesz solution we derive a 
detailed characterization of the solution of the spectral Laplacian obtained 
via  “harmonic lifting” approach. It is interesting to notice that in this 
setting, subject to the derived conditions related to the fractional power $s
$, the  Riesz and spectral solution are much closer. One possible explanation 
of this observation of the numerical tests is that there are no boundary 
layers in the considered particular test problems.

\section*{Acknowledgement}
The work has been partially supported by the Bulgarian National 
Science Fund under grant No. BNSF-DN12/1 and by the National Scientific Program "Information and Communication Technologies for a Single Digital Market in Science, Education and Security", financed by the Ministry of Education and Science. The research of N. Popivanov has been partially supported by the Bulgarian National Science Fund under grant No. DNTS-Russia 01/2/23.06.2017.

\end{document}